\documentclass[12pt]{article}
\usepackage{amsmath,amsthm,amssymb,amstext,dsfont,fancyvrb,float,fontenc,graphicx,subfigure,hyperref, tikz}
\usepackage[utf8]{inputenc}
\usepackage[margin=1in]{geometry}

\usepackage[english]{babel} 
\usepackage{mathtools}

\newtheorem{theorem}{Theorem}[section]

\newtheorem{lemma}[theorem]{Lemma}

\theoremstyle{definition}

\newtheorem{conjecture}[theorem]{Conjecture}
\newtheorem{question}[theorem]{Question}

\newcommand{\rb}{\text{rb}}
\newcommand{\Rb}{\text{Rb}}
\newcommand{\M}{\text{M}}
\renewcommand{\mod}[1]{\;(\text{mod }#1)}
\newcommand{\N}{\mathbb{N}}
\newcommand{\Z}{\mathbb{Z}}

\title{\Large\bf Bounds on Arithmetic Rainbow Ramsey Multiplicities}
\author{\sc Gabriel Elvin, Alexis Gonzales, Alejandro Rodriguez, \\ \sc and Israel Wilbur\thanks{This work was supported by CSUSB's Office of Student Research (OSR), the Proactive Approaches for Training in Hispanics in STEM grant (PATHS), and the Learning Aligned Employment Program (LAEP).}}

\begin{document}
\setcounter{page}{1}
\date{August 6, 2025}
\maketitle

\vskip 1.5em

\begin{abstract}
We study a quantitative Ramsey-type problem on $3$-term arithmetic progressions: how should the set of integers $[n] = \{1, 2, \dots, n\}$ be colored using 3 colors in order to maximize the number of rainbow 3-term arithmetic progressions? By \textit{rainbow}, we mean progressions whose elements are each assigned a distinct color. We determine a lower bound for this question and upper and lower bounds when $[n]$ is replaced with the integers modulo $n$, including an exact maximum when $n$ is a multiple of $3$.
\end{abstract}
 


\section{Introduction}

A classic starting point in Ramsey theory is the following.
Consider coloring the edges of a complete graph $K_n$ with $2$ colors, either red or blue.
What is the minimum $n$ such that no matter the coloring, the $K_n$ 
will contain a monochromatic triangle, i.e. a copy of $K_3$ that 
has either all red edges or all blue edges?
The answer is $6$, and importantly, if $K_3$ is replaced by any fixed graph $H$,
a monochromatic copy of $H$ always exists, as long as $n$ is made large enough,
a famous result known as Ramsey's Theorem \cite{R1928}. 
Since such an $n$ is guaranteed, we can ask a quantitative follow-up question:
for very large $n$, what is the minimum number of monochromatic copies of $H$ guaranteed
in a 2-coloring of the edges of a $K_n$?
Or, to normalize, what is the minimum \textit{proportion},
i.e. the number of monochromatic copies of $H$ divided by the total number of copies $H$ in the $K_n$?
A result of Goodman \cite{G1959} in 1959 answers this question for $H = K_3$:
as $n \to \infty$, the minimum proportion of monochromatic copies of $K_3$ approaches $1/4$.
Notably, if the edges of a $K_n$ were colored uniformly randomly and independently,
the expected proportion of monochromatic triangles also approaches $1/4$.
Therefore, perhaps surprisingly, independently coloring the edges of $K_n$ uniformly at random achieves the minimum proportion of 
monochromatic copies of $K_3$, at least asymptotically.

This question can be generalized or altered in several ways. First, to reiterate, any fixed graph $H$ can be investigated.
Second, any fixed number of colors can be used.
The asymptotic minimum for a graph is sometimes referred to as its \textit{Ramsey multiplicity}, or the proportion after normalizing its \textit{Ramsey multiplicity constant}, e.g. the Ramsey multiplicity constant of $K_3$ is 1/4.
These types of problems can be altered in another way: rather than attempt to minimize
the number of monochromatic copies of $H$, one can attempt to \textit{maximize} the number of \textit{rainbow} copies of $H$.
Here, $H$ is rainbow if every edge is assigned a different color.
Work in this area can be found in \cite{B2024, D2018, S2023}.

Further still, we need not be limited to graphs. 
Analogous to the question Goodman answered in 1959, the following was posed by Ron Graham (see \cite{RZ1998}) in 1996:
If each element of the set $[n] \coloneqq \{1, 2, \dots, n\}$ is assigned a color, either red or blue,
what is the minimum number of monochromatic solutions to the equation $x + y = z$?
By \textit{monochromatic solution}, we mean a triple $(x, y, x + y)$ such that either all elements are blue
or all are red.
Instinctively, one might guess that, analogous to Goodman's result, the minimum proportion would again be
$1/4$ asymptotically. However, in 1998, Robertson and Zeilberger \cite{RZ1998} showed that the minimum proportion, i.e.
\[\frac{\text{\# of monochr. solns. to $x + y = z$}}{\text{\# of solutions to $x + y = z$}},\]
actually approaches $2/11$ as $n \to \infty$\footnote{The authors did not
directly compute this quantity, but it follows immediately from their
result.}.

The correct analogy to Goodman's result 
is to replace $[n]$ with the integers modulo $n$, which we will denote by $\Z_n = \{0, \dots, n - 1\}$
(this analogy is in fact a correspondence; see \cite{SW2017} for more details).
In this setting, as shown by Datskovsky \cite{D2003} in 2003, the minimum proportion of monochromatic
solutions to $x + y = z$ is asymptotically $1/4$ (in fact, it is exactly $1/4$ when $n$ is even
and nearly so when $n$ is odd).

Just as with graphs, these questions can be expanded in many different directions by changing the equation,
the number of colors, or the set of possible inputs, or by flipping the question from ``minimizing monochromatic'' 
to ``maximizing rainbow''.
In general, we are interested in showing that, for a given $3$-term equation,
the maximum proportion of rainbow solutions (or minimum proportion of monochromatic solutions) 
is always asymptotically greater (or less) than what is expected from uniformly random colorings.
We do this by providing explicit constructions which give lower bounds on the maximum (or
upper bounds on the minimum).
Next
we introduce essential definitions and notation. 

By 3-term equation, we mean an equation of the form $ax + by = cz$, where $a, b, c$ are positive integers.
An \textbf{$r$-coloring} of a set $X$ is a surjective function $f : X \to [r]$.
A \textit{rainbow solution} to an equation is a triple $(x, y, z)$ that satisfies the equation 
and is such that each coordinate is assigned a different color, 
i.e. $f(x) \neq f(y)$, $f(x) \neq f(z)$, $f(y) \neq f(z)$.

Let $E$ be a shorthand for an equation $ax + by = cz$,
$X_n \in \{[n], \Z_n\}$, and $f$ be an $r$-coloring of $X_n$. Then
\begin{itemize}
    \item (solutions): $T_{X_n}(E) \coloneqq \{(x, y, z) \in {X_n}^3 \,:\, ax + by = cz\}$,
    \item (rainbow solutions): $\Rb_{f, X_n}(E) \coloneqq \{(x,y,z) \in T_{X_n}(E) \,:\, \text{$(x,y,z)$ is rainbow}\}$,
    \item (monochromatic solutions): 
    \[\M_{f, X_n}(E) \coloneqq \{(x,y,z) \in T_{X_n}(E) \,:\, f(x) = f(y) = f(z)\}\]
    \item (rainbow proportion): $\displaystyle \rb_{f, X_n}(E) \coloneqq \frac{|\Rb_{f, X_n}(E)|}{|T_{X_n}(E)|}$,
    \item (maximum rainbow proportion): $\displaystyle \rb_{X_n}(E) \coloneqq \max_{f : {X_n} \to [r]}\{\rb_{f}(E)\}$.
\end{itemize}
We will omit ``$X_n$'' from subscripts when the choice of $X_n$ is clear from context. Since our results are asymptotic, we will also cover the notation needed to precisely describe them.
Let $f$ and $g$ be functions of $n$, the asymptotic variable. If there exists constants $C$ and $N$ such that 
$|f(n)| \leq Cg(n)$ for all $n \geq N$, we say $f = O(g)$. By $f = \Omega(g)$, we mean $g = O(f)$.
Lastly, If $\lim_{n\to\infty}f/g = 0$, then we say $f = o(g)$.

In what follows, we restrict our analysis to $3$-term equations and $3$-colorings.
Our main goal is to show that for certain equations, it is possible to find (sequences of) colorings that produce asymptotically more rainbow
solutions than what is expected from uniformly random colorings.
We capture this characteristic with the following definition:
$E$ is \textbf{$r$-rainbow-uncommon} over $X_n$ if 
\begin{equation}
\liminf_{n\to\infty} \rb_{X_n}(E) = \frac{2}{9} + \Omega(1).
\end{equation}
Note that $2/9$ is the (asymptotic) expected proportion of rainbow solutions to an equation
when coloring uniformly at random:
if a solution $(x,y,z)$, with $x,y,z$ distinct\footnote{The
solutions where $x,y,z$ are \textit{not} all distinct becomes a vanishing error term asymptotically.
}, is colored with 3 colors, there are $6$ rainbow colorings
out of $27$ total colorings, and the asymptotic $2/9$ proportion follows.
The $\Omega(1)$ term indicates we must be strictly above the $2/9$ random threshold.
A question regarding this definition is if
an asymptotic maximum always exists, i.e. if ``$\liminf$'' can be replaced with ``$\lim$''.
This has been settled for analogous questions on graphs 
(the answer is ``yes'' \cite{D2018}), 
but it is still open in the realm 
of equations.

Our main result centers around 3-term arithmetic progressions, which we abbreviate as $3$-APs.
We utilize the fact that $3$-APs are in 1-1 correspondence
with solutions to the equation $x + y = 2z$:
a progression $\{a, a + d, a + 2d\}$ corresponds to the solution $(a, a + 2d, a + d)$,
and a solution will always be in arithmetic progression.
Note that we consider $(x,y,z)$ and $(y,x,z)$ as distinct solutions,
and this amounts to reversing the order of the progression,
which we also count as distinct.
We now state our main result.
\begin{theorem}\label{thm-main}
The equation $x + y = 2z$ is $3$-rainbow-uncommon over both $[n]$ and $\Z_n$.
In particular,
\begin{align}
\frac{2}{3} + o(1) &\leq \rb_{[n]}(x + y = 2z), \label{result-n}\\
\frac{1}{3} - o(1) &\leq \rb_{\Z_n}(x + y = 2z) \leq \frac{2}{3}, \text{ when $n \neq 3t$,} \label{result-Zn-not}\\
&\;\;\;\;\; \rb_{\Z_n}(x + y = 2z) = \frac{2}{3}, \text{ when $n = 3t$,} \label{result-Zn-is}
\end{align}
where $t \in \N$.
\end{theorem}
To prove the lower bounds, we find constructions that yield enough rainbow solutions
and explicitly count the number of such solutions. To prove the upper bounds, we use a straightforward
counting and optimization technique.

We also obtain an additional result for the equation $x + y = z$ over $\Z_n$.
\begin{theorem}
$x + y = z$ is 3-uncommon over $\Z_n$. 
\end{theorem}


\section{Proof of Main Result}

Before analyzing any colorings, we present the following lemma,
which describes how many total solutions a specific class of equations has over $[n]$.

\begin{lemma}\label{lem-soln-n}
For $c \geq 2$, there are $n^2/c + O(n)$ solutions to $x + y = cz$ over $[n]$.
\end{lemma}

\begin{proof}
Let $n = ck + r$ $(0 \leq r \leq c - 1)$. Then $[ck + r] = [ck] \cup \{ck + 1,\ldots, ck + r\}$. We show there are $n^2/c - O(n)$ solutions in $[ck]$ and $O(n)$ solutions with a coordinate in $\{ck + 1,\ldots, ck + r\}$. First, we partition $[ck]$ into $c$ disjoint subsets, $R_0,\ldots,R_{c-1}$, where each $R_i$ has size $(n - r)/c$ and consists of all elements whose remainder modulo $c$ is $i$. For each remainder $i \in \{0,\ldots,c-1\}$, there exists the unique remainder $c - i$ such that when any element $x \in R_i$ is added to any element $y \in R_{c-i}$, the triple $(x, y, \frac{x + y}{c})$ is a solution. Therefore, each of the $c$ subsets will correspond to $((n - r)/c)^2$ solutions, for a total of \begin{equation}\label{eq-solns-[ck]}
    c \cdot \frac{(n-r)^2}{c^2} = \frac{(n-r)^2}{c} = \frac{1}{c}n^2 - O(n)
\end{equation} solutions in $[ck]$ (recall that we are counting $(x, y, z)$ and $(y, x, z)$ as distinct solutions). 

Lastly, we bound the number of solutions which include at least one element of $\{ck + 1,\ldots, ck + r\}$. Fix $x$ in this set. If $y$ is chosen next from $[n]$, $z$ is now determined (or possibly not in $[n]$).
With $r$ choices for $x$ and at most $n$ choices for $y$, this yields no more than $rn$ solutions, and with 6 orderings in which to choose $x,y,z$, this results in at most $6rn$ solutions, giving us the desired result.
\end{proof}

Next, we describe the coloring used to achieve our main result. 
Define the $3$-coloring $f\colon [n] \to \{1, 2, 3\}$ by \begin{align}\label{3-coloring-func}
f(x) &= 
\begin{cases}
1 \text{ (red)}, & x \equiv 1 \mod 3, \\
2 \text{ (blue)}, & x \equiv 2 \mod 3, \\
3 \text{ (green)}, & x \equiv 0 \mod 3.
\end{cases}
\end{align} 

We now prove (\ref{result-n}).

\begin{proof}
Let $E$ be the equation $x+y = 2z$. We obtain a lower bound for the maximum via the coloring $f$. Under $f$, $|T_{[n]}(E)| = |\Rb_{f}(E)| + |\M_{f}(E)|$; that is, there are no dichromatic solutions. To see this, if $x \equiv a \mod 3$, then $z \equiv a + d \mod 3$ and $y \equiv a + 2d \mod 3$, where $d$  is the common difference of the corresponding $3$-AP. For each possible $d \equiv 0, 1, 2 \mod 3$, 
the remainders of $x$, $y$, and $z$ mod $3$ are either all equal (when $d \equiv 0 \mod 3$) or all distinct (when $d \equiv 1, 2 \mod 3$). 
By Lemma \ref{lem-soln-n}, $|T_{[n]}(E)| = n^2/2 + O(n)$, so determining the number of monochromatic solutions will yield $|\Rb_{f}(E)|$, leading to the proportion in (\ref{result-n}). Note that a pair $(x, y)$ forms 
a valid solution to $E$ and satisfies $x \equiv y \mod 3$ if and only if $x \equiv y \mod 6$. Since each residue class modulo 6 contains $n/6 + O(1)$ elements, the number of solutions satisfying $x \equiv y \mod 6$ (and thus $x \equiv y \mod 3$) is 
\[6 \cdot \left(\frac{n}{6} + O(1)\right)^2 = \frac{n^2}{6} + O(n).\] 
As these are the monochromatic solutions, it follows that $|\Rb_{f}(E)| = n^2/3 + O(n)$. Therefore, \begin{align}
    \rb_{[n]}(E) \geq \displaystyle \rb_{f}(E) = \frac{n^2/3 + O(n)}{n^2/2 + O(n)} = \frac{2}{3} + o(1),
\end{align} which is our desired result.
\end{proof}

For our result over $\Z_n$, we will use three classic results from elementary number theory.
We state them here but omit their proofs and then use them to prove Lemma \ref{lem-solns-Zn} below.

\begin{theorem}\label{thm-nt-1}
The congruence $ax \equiv b \mod{n}$ is solvable in integers if and only if $\gcd(a,n)|b$.
\end{theorem}

\begin{theorem}\label{thm-nt-2}
    Given integers $a$, $b$, and $c$ with $a$ and $b$ not both $0$, there exist integers $x$ and $y$ that satisfy the equation $ax + by  = c$ if and only if $\gcd(a,b)|c$.
\end{theorem}

\begin{theorem}\label{thm-nt-3}
If $ax \equiv b \mod{n}$ has a solution, then there are exactly $\gcd(a,n)$ solutions in $\{0, 1, 2, \ldots, n - 1\}$, the canonical complete residue system modulo $n$.
\end{theorem}

\begin{lemma}\label{lem-solns-Zn}
There are $n^2$ solutions to $ax + by = cz$ over $\Z_n$, assuming $\gcd(a,b,c) = 1$.
\end{lemma}

\begin{proof}
        We wish to count the number of solutions to
            \begin{equation}\label{alex-1}
                ax + by \equiv cz \mod{n}, \qquad 0 \leq x, y, z \leq n - 1.
            \end{equation}
        For a fixed $x, y$, by Theorem \ref{thm-nt-1}, there exists $z$ which satisfies the congruence (\ref{alex-1}) if and only if
            \begin{equation}\label{alex-2}
                \gcd(c,n) | ax + by.
            \end{equation}
        Putting $d = \gcd(c,n)$, (\ref{alex-2}) gives us
            \begin{equation}\label{alex-3}
                \begin{split}
                    ax + by &= pd,\\
                    pd - by &= ax,
                \end{split}
            \end{equation}
        where $p \in \Z$. Now by Theorem \ref{thm-nt-2}, a $y$ which satisfies (\ref{alex-3}) exists if and only if $\gcd(b,d)$ divides $ax$. Note that
        this is actually if and only if $\gcd(b,d) | x$, since $\gcd(b,d)$ and $a$ are relatively prime,
        as we are assuming the original equation $ax + by = cz$ is fully reduced. 
        Thus, there are $n /\gcd(b,d)$
        many $x$s in $\Z_n$ that will be a multiple of $\gcd(b,d)$ and hence yield a solution. 
        Now suppose $x$ yields a solution, i.e. satisfies $\gcd(b,d)|x$. Taking (\ref{alex-3}) modulo $d$ gives us
            \begin{equation}\label{alex-5}
                -by \equiv ax \mod{d}.
            \end{equation}
        Then by Theorem \ref{thm-nt-3} there are $\gcd(b,d)$ many $y$s in $\Z_d$ that will satisfy (\ref{alex-5}). Note, each solution in $\Z_d$ corresponds to $n/d$ solutions in $\Z_n$. Thus, there will be $\gcd(b,d) \cdot (n/d)$ suitable $y$s in $\Z_n$ for each suitable $x$. Therefore, there will be
            $$\left( \frac{n}{\gcd(b,d)} \right) \cdot \left( \frac{\gcd(b,d) \cdot n}{d} \right) = \frac{n^2}{d}$$
        pairs $(x,y)$ which satisfy (\ref{alex-2}). Again by Theorem \ref{thm-nt-3}, for each such pair there will be exactly $\gcd(c,n) = d$ many $z$s in $\Z_n$ that form a solution to (\ref{alex-1}). Therefore, there are
            $(n^2 / d) \cdot d = n^2$ total solutions.
\end{proof}

Now we will prove the lower bounds of (\ref{result-Zn-is}) and (\ref{result-Zn-not}), respectively.
In what follows, we refer to $\Z_n$ as the set $\{\overline{0}, \overline{1}, \dots, \overline{n-1}\}$.

\begin{proof} 

        First, we redefine the coloring $f$ (note it has the same basic structure as in (\ref{3-coloring-func})).
        Let $\iota : \Z_n \to \{0,1,\dots,n-1\}$ be simply $\iota(\overline{x}) = x$: each element gets mapped to its unique representative in $\{0,1,\dots,n-1\} \subseteq \Z$, and let $g : \{0,1,\dots,n-1\} \to \{1, 2, 3\}$ be
        \begin{equation}
        g(x) = 
\begin{cases}
1 \text{ (red)}, & x \equiv 0 \mod 3, \\
2 \text{ (blue)}, & x \equiv 1 \mod 3, \\
3 \text{ (green)}, & x \equiv 2 \mod 3.
\end{cases}
        \end{equation}
        Finally, let $f = g \circ \iota$. Suppose we have a solution $(\overline{x}, \overline{y}, \overline{z}) = (\overline{a}, \overline{a} + 2\overline{d}, \overline{a} + \overline{d})$. The following argument is simpler over the integers, so we will use $\iota(\cdot)$ throughout, e.g. $d = \iota(\overline{d}) \in \Z$. Just as in the proof of (\ref{result-n}), for each possible $d \equiv 0, 1, 2 \mod 3$, 
the remainders of $x$, $y$, and $z$ mod $3$ are either all equal (when $d \equiv 0$) or all distinct (when $d \equiv 1, 2$).
Crucially, when $n = 3t$, a solution is rainbow if and only if $d \equiv 1 \mod{3}$ or $d \equiv 2 \mod{3}$, just as over $[n]$.
This is because when the solution is viewed as a $3$-AP moving sequentially $a \to a + d \to a + 2d$, if $a + d$ or $a + 2d$ go beyond $n$
and wrap back around,
the pattern repeats exactly (red, blue, green, \dots) and nothing changes.
Since each type of $d$ is equally likely, in this case, exactly $2/3$ of the solutions will be rainbow.

    Now we consider the case when $n \neq 3t$. To do this, we again make use of the $1$-$1$ correspondence between $3$-APs and solutions to the equation $x + y = 2z$. Consider each $3$-AP whose common difference satisfies $0 \leq d \leq \lfloor n/2 \rfloor$ and $d \neq 3t$. We only consider these because any $3$-AP with common difference greater than $\lfloor n/2 \rfloor$ will have the negative common difference of an already counted progression. Moreover, if $d = 3t$, then by definition of $f$, the $3$-AP will not be rainbow unless the progression could wrap around twice, but this will not happen because of the restriction $d \leq \lfloor n/2 \rfloor$. This will allow us to obtain the total number of rainbow solutions by counting only those that satisfy $0 \leq d \leq \lfloor n/2 \rfloor$ and essentially doubling the result (The cases where $d = 0$ and $d = n/2$ when $n$ is even get counted twice, resulting in a negligible $O(1/n)$ error term after dividing by the total number of progressions). Each $3$-AP corresponds to a pair $(x, d)$ where $x$ is the first term and $d$ is the common difference. Note that such a $3$-AP is rainbow if and only if its corresponding pair satisfies $x + 2d < n$. To see this, first recall that if $x + 2d \geq n$, the 3-AP wraps around only once because of the upper bound on $d$, and this type of progression is rainbow if and only if $n$ is a multiple of 3, which we are assuming is not the case. In Figure \ref{fig:graph}, we see that this inequality forms a triangle of area $n^2/4$, so it contains $n^2/4 - O(n)$ 3-APs (as circles/dots in the figure). Furthermore, $2/3 - o(1)$ of those $3$-APs have $d$ not a multiple of $3$. 
    \begin{figure}
        \centering
\begin{tikzpicture}[scale=1.2]
\filldraw[gray!50] (0,2) -- (4,0) -- (0,0) -- cycle;
\draw (0,2) node[anchor=east]{$\frac{n}{2}$};

\draw[thick] (-0.5,0) -- (4.5,0) node[anchor=west]{$x$};
\draw[thick] (0,-0.5) -- (0,2.5) node[anchor=south]{$d$};
\draw (4,0) node[anchor=north]{$n$};
\foreach \x in {0,...,40}{
    \foreach \d in {1,4,...,19}{
        \filldraw (0.1*\x,0.1*\d) circle (0.03);
    }
    \foreach \d in {2,5,...,20}{
        \filldraw (0.1*\x,0.1*\d) circle (0.03);
    }
    \foreach \d in {0,3,...,18}{
        \draw[black!50] (0.1*\x,0.1*\d) circle (0.03);
    }
}
\draw[dashed] (4.075,0) -- (4.075,2.075) -- (0,2.075);

\end{tikzpicture}
        \caption{Graph of $x + 2d < n$ shaded with each pair $(x, d)$ representing a $3$-AP that is black if $d \neq 3t$ and white if $d = 3t$.}
        \label{fig:graph}
    \end{figure}
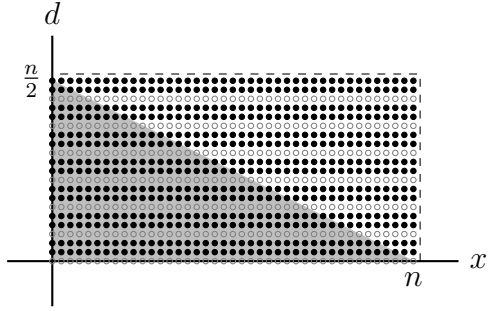
    Therefore, \begin{equation}
        \left(\frac{1}{4}n^2 - O(n)\right) \cdot \left(\frac{2}{3} - o(1)\right) = \frac{1}{6}n^2 - O(n)
    \end{equation} of the $3$-AP's are rainbow. Doubling the result to account for $d > \lfloor n/2 \rfloor$
    and dividing by $n^2$, the total number of solutions given by Lemma \ref{lem-solns-Zn}, completes the proof.
    \end{proof}

Finally, we prove the upper bounds of (\ref{result-Zn-not}) and (\ref{result-Zn-is}).

\begin{proof}
We claim that at least $1/3$ of solutions cannot be rainbow.
Let $rn$, $bn$, $gn$, be the number of red, blue, and green elements of $\Z_n$, respectively,
where, $r, b, g \geq 0$ and $r + b + g = 1$.
Note that every pair $(x, z)$ uniquely determines a solution to $x + y = 2z$.
Therefore, we can count the number of monochromatic pairs of this form to get a lower bound on the 
number of non-rainbow solutions.
The number of non-rainbow solutions is at least
\begin{align*}
& \text{\# of red pairs $(x,z)$} + \text{\# of blue pairs $(x,z)$} + \text{\# of green pairs $(x,z)$} \\
=& (rn)^2 + (bn)^2 + (gn)^2 \\
=& (r^2 + b^2 + g^2)n^2.
\end{align*}
For a fixed $n$, the final expression is minimized when $r = b = g = 1/3$, which means at least $n^2/3$ 
solutions will not be rainbow, so at most $2n^2 / 3$ will be rainbow, and our initial claim follows after applying Lemma \ref{lem-solns-Zn}.
Note that the lower and upper bounds of $\rb_{\Z_n}(x + y = 2z)$ coincide when $n = 3t$, giving us a true maximum in (\ref{result-Zn-is}).
\end{proof}

\section{An Extra Result}\label{sec-extra}

We also obtain a ``traditional'' uncommonness result. 
That is, a result regarding the \textit{minimum} number of \textit{monochromatic} solutions.
The work in this section serves an example of how to approach a modified version of the original problem.

Note that for a $3$-term equation, when coloring uniformly at random using $3$ colors, a solution (with distinct inputs) has a
$1/9$ probability of being monochromatic. 
Therefore, to show a particular $3$-term equation is $3$-uncommon, we can find a coloring 
with asymptotically strictly \textit{fewer} than a $1/9$ proportion of monochromatic solutions.

\begin{theorem}\label{thm-alexis}
$x + y = z$ is $3$-uncommon over $\Z_n$.
In particular,

\begin{equation}\label{eq-schur-bounds}
\frac{\text{min. \# of monochr. solutions}}{\text{total \# of solutions}} \leq 
\begin{cases}
\frac{1}{25}, & n = 5t, \\
\frac{1}{10} + o(1), & n \neq 5t
\end{cases}.
\end{equation}
\end{theorem}

\begin{proof}
Let $\iota : \Z_n \to \{0,1,\dots,n-1\}$ be defined as in the proof of (\ref{result-Zn-not}) and (\ref{result-Zn-is}), and let $g : \{0,1,\dots,n-1\} \to \{1,2,3\}$ be:
\begin{equation}
g(x) =
\begin{cases}
1 \text{ (red)}, & x \equiv 0 \mod 5, \\
2 \text{ (blue)}, & x \equiv \pm 1 \mod 5, \\
3 \text{ (green)}, & x \equiv \pm 2 \mod 5.
\end{cases}
\end{equation}
Then $f = g \circ \iota$ achieves the bounds. To show this, we split into cases depending on the value of $n \mod 5$.

First, if $n = 5t$, then it follows that the only monochromatic solutions are red. 
In this case, each red pair $(x, y)$ uniquely determines a red solution $(x, y, x + y)$, 
and hence there are 
\begin{equation}
\left(\frac{n}{5}\right)\left(\frac{n}{5}\right) = \frac{1}{25}n^2
\end{equation}
monochromatic solutions. 
Dividing by the total number of solutions
provided by Lemma \ref{lem-solns-Zn}, we get the first case of (\ref{eq-schur-bounds}). 

For $n \neq 5t$, a red pair $(x,y)$ will lead to a monochromatic solution if and only if $x + y < n$,
which is the case for $1/2 - o(1)$ of the total pairs.
Therefore, since there are $(\lceil n/5 \rceil)^2 = n^2/25 + O(n)$ red pairs, there are
\begin{equation}
\left( \frac{1}{2} - o(1) \right) \cdot \left(\frac{n^2}{25} + O(n)\right) = \frac{1}{50}n^2 + O(n)
\end{equation}
red solutions.
Now we split into two cases to count blue and green solutions.

If $n = 5t \pm 1$, a blue pair $(x,y)$ can lead to a monochromatic solution only 
when $x + y \geq n$, which happens roughly half the time. And among these pairs, 
we will get a blue solution in the following cases:
\begin{itemize}
    \item When $x \equiv 1 \mod 5$ and $y \equiv -1 \mod 5$, 
    \item When $x \equiv -1 \mod 5$ and $y \equiv 1 \mod 5$,
    \item If $n = 5t + 1$, then when $x, y \equiv 1 \mod 5$, or if $n = 5t - 1$, then when $x, y \equiv -1 \mod 5$.
\end{itemize}
Regardless of whether $n = 5t + 1$ or $n = 5t - 1$, three out of the four possibilities for $(x,y) \in \{\pm 1\}^2$ (modulo 5) yield blue solutions. Because the possibilities are equally likely, $3/4 - o(1)$ of the viable blue pairs will yield blue solutions (the error term here comes from the fact that the number of blue pairs is not divisible by 4 in this case).
Therefore, since there are $(\lceil 2n/5 \rceil)^2 = 4n^2/25 + O(n)$ blue pairs, there will be 
\begin{equation}
\left(\frac{1}{2} - o(1)\right)\cdot\left(\frac{3}{4} - o(1)\right)\cdot\left(\frac{4n^2}{25} + O(n)\right) = \frac{3}{50}n^2 + O(n)
\end{equation}
blue solutions.
Counting green solutions here is similar, the difference being that now only $1/4$ of the
viable pairs $(x,y)$ will lead to a green solution: $x,y \equiv 2 \mod 5$ when $n = 5t + 1$ and $x,y \equiv -2 \mod 5$ when $n = 5t - 1$.
Finally adding the red, blue, and green solutions, we get
\begin{equation}
\left(\frac{1}{50} + \frac{3}{50} + \frac{1}{50}\right)n^2 + O(n) = \frac{1}{10}n^2 + O(n)
\end{equation}
monochromatic solutions.

The case where $n = 5t \pm 2$ is similar to the previous one but with the blue and green cases swapped, so we omit the computations.
Dividing by $n^2$, the total number of solutions, proves the theorem.
\end{proof}

\section{Future Work}
We believe that the lower bounds of (\ref{result-n}) and (\ref{result-Zn-not}) are the true maxima, but these are still open questions.
We will not speculate on whether the bounds in Theorem \ref{thm-alexis} are optimal, but it would interesting to investigate further and either find a better coloring or show that our coloring is the best possible.
The possibilities for future work are numerous.
Any of the following aspects of the problem can be altered, generating new questions:
the number of colors used, the equation/progression being studied,
flipping between ``maximizing rainbow'' and ``minimizing monochromatic'',
the set of inputs, and more.
Below are several open questions ready for exploration.
\begin{question}
Does $\displaystyle\lim_{n\to\infty}\rb_{[n]}(E)$ always exist? What about in the ``minimizing monochromatic'' setting?
\end{question}
\begin{conjecture}[\cite{TW2017}]
$\displaystyle\lim_{n\to\infty}\rb_{[n]}(x + y = z) = \frac{2}{5}$.
\end{conjecture}
\begin{question}
What are some bounds on the maximum number of rainbow 3-APs when more than 3 colors are used?
\end{question}

\section*{Acknowledgments} 
 
We would like to thank Cal State San Bernardino's 
Undergraduate Summer Research Program for partially funding 
this research. In particular, during Summer 2023 the first author was
funded through Cal State San Bernardino's Office of Student Research
and the remaining authors received funding from the
Proactive Approaches for Training in Hispanics in STEM grant (PATHS).
Furthermore, during Fall 2023 all but the first author were funded through 
the Learning Aligned Employment Program (LAEP), 
and the third author was funded  through LAEP through Spring 2024.

\newpage 

{\footnotesize
\bibliographystyle{amsplain}
\bibliography{in_doc}}


 
{\footnotesize  
\medskip
\medskip
\vspace*{1mm} 
 
\noindent {\it Gabriel Elvin}\\  
California State University, San Bernardino\\
5500 University Pkwy\\
San Bernardino, CA 92407\\
E-mail: {\tt Gabriel.Elvin@csusb.edu}\\ \\ 

\noindent {\it Alexis Gonzales}\\  
California State University, San Bernardino\\
5500 University Pkwy\\
San Bernardino, CA 92407\\
E-mail: {\tt 006840256@coyote.csusb.edu}\\ \\  

\noindent {\it Alejandro Rodriguez}\\  
California State University, Fullerton \\
800 N State College Blvd\\
Fullerton, CA 92831 \\
E-mail: {\tt alejandro.rod98@csu.fullerton.edu}\\ \\

\noindent {\it Israel Wilbur}\\  
University of Georgia \\
University of Georgia Chapel, Herty Dr\\ 
Athens, GA 30602\\
E-mail: {\tt irw38277@uga.edu} \\ \\   

}

\end{document}